\documentclass{article}
\usepackage{amssymb} 
\usepackage{amsthm} 
\usepackage{amsmath} 
\usepackage{float} 
\usepackage{fancyhdr} 
\usepackage[english]{babel}
\usepackage[utf8]{inputenc}
\usepackage[T1]{fontenc}
\usepackage[margin=1.5in]{geometry} 
\usepackage{wasysym}

\newtheorem{theorem}{Theorem}[section]

\usepackage[osf,sc]{mathpazo}
\usepackage{textcomp}
\usepackage{microtype}
\usepackage{graphicx, color}
\usepackage{amsfonts}
\usepackage{booktabs}
\usepackage{pifont}
\usepackage{tikz}
\usepackage{graphicx}
\usepackage{float}
\usepackage{hyperref}
\usepackage[all]{xy}

\begin{document}
\title{Answer to a question concerning Euler's paper {}"Variae considerationes circa series hypergeometricas"{}}
\author{Alexander Aycock}
\date{ }
\maketitle

\begin{abstract}
We solve a problem concerning Euler's paper "Variae considerationes circa series hypergeometricas"{} (\cite{E661}), as suggested by G. Faber in the preface to Volume 16,2 of the first series of Euler's Opera Omnia. Our solution employs methods introduced by Euler at other places.
\end{abstract}

\section{Introduction}
\label{sec: Introduction}

This paper is about Euler's work "Variae considerationes circa series hypergeometricas"{} (\cite{E661}). In that paper, Euler introduced the following function,\footnote{In the same paper, Euler also introduced two other related functions. But since they will not be needed for our purposes, we will not discuss them in the following.}  presented here in Euler's notation:

\begin{equation}\label{eq: Euler'sche Darstellung}
     \Gamma:i = a\cdot (a+b) \cdot (a+2b) \cdot (a+3b)\cdots (a+(i-1)b).
\end{equation}
Using the Euler-Maclaurin summation formula\footnote{Euler obtained his version of the Euler-Maclaurin summation formula in his treatise \cite{E47}.}, Euler arrived at the following expression -- in his notation -- for the function given in Eq.\,(\ref{eq: Euler'sche Darstellung})

\begin{equation}\label{eq: Asymptotiken}
     \Gamma:i = A (a-b+bi)^{\frac{a}{b}+i-\frac{1}{2}}e^{-i} .
\end{equation}
Here, Euler understood $i$ as an infinitely large positive number such that the previous equation has to be understood as an asymptotic equation from a modern perspective. The constant $A$ cannot be defined by means of the Euler-Maclaurin summation formula and Euler was not able  to find it in any other way\footnote{The author of this paper was not able to locate a source in which this constant is actually evaluated.}.
This induced G. Faber (in the introduction to volume 16,2 of the first series of Euler's Opera Onmia) to the following statement:\\[2mm]
\textit{"{}It might perhaps be worthwhile to take up again Euler's efforts to evaluate this number $A$".}\footnote{"Es würde sich vielleicht lohnen, die Euler'schen Bemühungen um die Ermittelung dieser Zahl $A$ wieder aufzunehmen." \cite[p.xliv]{Fa35}}\\[2mm]
In the following we shall present a solution to this task. In doing so, we intend to use only methods and ideas that were explicitly explained by Euler at some place.

\section{Solution of the Task}
\label{sec: Solution of the Task}

\subsection{Finding a relation among the functions $\Gamma_{E}$ and $\Gamma$}
\label{subsec: Herstellung einer Beziehung zwischen GammaE und Gamma}

In the following, we will denote Euler's function $\Gamma:i$ by the symbol $\Gamma_E(i)$ to distinguish it from the function that is indicated by $\Gamma$ nowadays (see equation (\ref{eq: Gamma-Funktion})). The main step towards the solution of the task at hand consists in relating the Eulerian function $\Gamma_E$ (equation (\ref{eq: Euler'sche Darstellung})) to the familiar $\Gamma$-function (equation (\ref{eq: Gamma-Funktion})). The $\Gamma$-function is nowadays usually defined by the expression

\begin{equation}\label{eq: Gamma-Funktion}
    \Gamma(x):= \int\limits_{0}^{\infty}t^{x-1}e^{-t}dt, \quad \text{for} \quad \operatorname{Re}(x)>0.
\end{equation}
Clearly, for integer values of $x$, $\Gamma(x)$ is a special case of $\Gamma_E(x)$ for $a=b=1$.
For the $\Gamma$-function an asymptotic expansion is provided by the Stirling formula. The final relation we are after is then found in Theorem \ref{Theorem: Beziehung}. 

The first step, therefore, will be to derive an integral representation for the function $\Gamma_E(x)$, using a method introduced by Euler in his papers \cite{E123} and \cite{E594}, papers actually devoted to the theory of continued fractions. In \cite[\S13, Ex.13]{E594}, Euler also arrived at the modern expression for the $\Gamma$-function. However, in most of his investigations Euler preferred the expression 

\begin{equation*}
    \int\limits_{0}^1 \log \left(\dfrac{1}{t}\right)^{x}dt,
\end{equation*}
which he discovered in \cite[\S14]{E19} and used for the interpolation of the factorial. By the substitution $\ln(\frac{1}{t})=u$, Euler's preferred expression reduces to the integral expressing $\Gamma(x+1)$. The notation $\Gamma$ for the integral in equation $(\ref{eq: Gamma-Funktion})$ was introduced by Legendre in his paper \cite{Le09} and popularised by his book \cite{Le26}.
%
%

We first intend to prove the following:

\begin{theorem}\label{Theorem: Integraldarstellung}
The Eulerian function $\Gamma_E$ has the following integral representation:

\begin{equation*}
    \Gamma_E(x)= \dfrac{b^x}{\Gamma\left(\frac{a}{b}\right)}\int\limits_{0}^{\infty}y^{x-1+\frac{a}{b}}e^{-y}dy \quad \text{for} \quad \operatorname{Re}\left(x-1+\frac{a}{b}\right)>0.
\end{equation*}
\end{theorem}

\begin{proof}
From Euler's representation of  $\Gamma_E(x)$ (equation (\ref{eq: Euler'sche Darstellung})) we deduce the following functional equation

\begin{equation} \label{eq: Funktionalgleichung}
    \Gamma_E(x+1)=(a+bx)\Gamma_E(x),
\end{equation}
together with the initial condition $\Gamma_E(1)=a$. This functional equation interpolates the Eulerian expression for $\Gamma_E(x)$, which is only valid for natural numbers $x$. Furthermore, also based on the representation (\ref{eq: Euler'sche Darstellung}), we will assume that  $\Gamma_E(x)$ is logarithmically convex for positive real numbers $x $ satisfying $\operatorname{Re}\left(x-1+\frac{a}{b}\right)>0$. Then we can apply the above-mentioned Eulerian method of solving homogeneous difference equations with linear coefficients. 

According to this method Eq.\,(\ref{eq: Funktionalgleichung}) is solved by an expression of the form

\begin{equation}\label{eq: Ansatz}
    \int\limits_{c}^{d} y^{x-1}P(y)dy,  
\end{equation}
where the method also teaches how to find the boundaries of integration $c$ and $d$ and the function $P(t)$, which is assumed to be differentiable. \\[2mm]
In order to demonstrate the application of that method to the example at hand, let us consider the following auxiliary equation:

\begin{equation} \label{eq: Hilfsgleichung}
    \int^y y^{x}P(y)dy= (a+bx)\int^y y^{x-1}P(y)dy+ y^{x}Q(y),
\end{equation}
where $Q(y)$ is another differentiable function not known at this moment. By $\int^y f(y)dy$ we denote the integral function of $f$, i.e., if $F(y)$ is a primitive of $f(y)$, then $\int^y f(y)dy:= \int\limits_{y_0}^y f(t)dt=F(y)-F(y_0)$, where $y_0$ is to be chosen in such a way that $F(y_0)=0$. \\ Differentiating and dividing equation (\ref{eq: Hilfsgleichung}) through by $y^{x-1}$, we arrive at the following equation: 

\begin{equation} \label{eq: Hilfsgleichung2}
    yP(y)= (a+bx)P(y)+xQ(y)+yQ'(y),
\end{equation}
where, without loss of generality, $y \neq 0$ was assumed. Comparing coefficients of powers of $x$ on each of side of the equation, will lead us to the following system of differential equations for the functions $P(y)$ and $Q(y)$:


\begin{alignat}{2}
     yP(y) &= a P(y) + y&&Q'(y), \\
     0       &= b P(y)  \;+  &&Q(y).  
\end{alignat}

The solutions of this system are:

\begin{equation*}
    P(y) =-\dfrac{C}{b}e^{-\frac{y}{b}}y^{\frac{a}{b}} \quad \text{and} \quad Q(y)=Ce^{-\frac{y}{b}}y^{\frac{a}{b}},
\end{equation*}
where $C\neq 0$ is a constant of integration.\\[2mm]
Going back to the auxiliary equations (\ref{eq: Hilfsgleichung}) and (\ref{eq: Hilfsgleichung2}), we conclude that an ansatz of the form (\ref{eq: Ansatz}) will solve the difference equation (\ref{eq: Funktionalgleichung}) if we fix the the boundaries of integration, i.e., $c$ and $d$,  from solutions to the equation:
\begin{equation}
\label{eq: Boundaries}
    d^xQ(d)-c^xQ(c) =0.
\end{equation}
Under the conditions $\operatorname{Re}b>0$ and $\operatorname{Re}\left(x+\frac{a}{b}\right)>0$, aside from the trivial solution $c=d$, we find the pair\footnote{It would be interesting to investigate, whether there are even more solutions to equation (\ref{eq: Boundaries}), but since we only need one non-trivial solution, we will not consider this question here.}

\begin{equation*}
    (c,d)=(0,\infty).
\end{equation*}
Inserting everything we discovered into the ansatz (\ref{eq: Ansatz}), we find:

\begin{equation*}
    \Gamma_E(x)=-\dfrac{C}{b}\int\limits_{0}^{\infty}y^{x-1}e^{-\frac{y}{b}}y^{\frac{a}{b}}dy.
\end{equation*}
It remains to evaluate the constant $C$. For this, we will need the initial condition $\Gamma_E(1)=a$. Via the last equation and the initial condition we find:

\begin{equation*}
    \left(\Gamma_E(1)=\right)a= -\dfrac{C}{b}\int\limits_{0}^{\infty}y^{\frac{a}{b}}e^{-\frac{y}{b}}dy.
\end{equation*}
The integral on the right-hand side can be expressed using the $\Gamma$-function (\ref{eq: Gamma-Funktion}) by a simple rescaling of the integration variable such that we eventually arrive at the equation:

\begin{equation*}
    C=-\dfrac{ab^{-\frac{a}{b}}}{\Gamma\left(\frac{a}{b}+1\right)}.
\end{equation*}
Because of the relation $\Gamma(x+1)=x\Gamma(x)$ this simplifies even further to the expression:

\begin{equation*}
    C=-\dfrac{b^{-\frac{a}{b}+1}}{\Gamma\left(\frac{a}{b}\right)}.
\end{equation*}
If we insert this value of the constant of integration $C$ into the ansatz (\ref{eq: Ansatz}), we just have to perform another substitution $\frac{y}{b}\mapsto y$ in the integral. After a quick calculation, eventually we arrive at Theorem \ref{Theorem: Integraldarstellung} which we wanted to prove. 
\end{proof}

\subsubsection{Relation among $\Gamma$ and $\Gamma_E$}
\label{subsubsec: Beziehung zwischen Gamma und GammaE}

Via \ref{Theorem: Integraldarstellung} the function $\Gamma_E(x)$ can be expressed in terms of the function $\Gamma(x)$. More precisely, the integral in \ref{Theorem: Integraldarstellung} can be expressed by using the definition of the $\Gamma$-function (\ref{eq: Gamma-Funktion}) such that we have the following

\begin{theorem}\label{Theorem: Beziehung}
The following equation holds:

\begin{equation*}
    ´\Gamma_E(x)= \dfrac{b^x}{\Gamma\left(\frac{a}{b}\right)}\cdot\Gamma \left(x+\dfrac{a}{b}\right).
\end{equation*}
\end{theorem}
This relation will be useful for the determination of the constant $A$, which we will do next.

\subsection{Determination of the constant $A$}
\label{subsec: Determination of the constant A}

To determine the constant $A$ we will need the Stirling formula, i.e., in modern terms the following asymptotic:

\begin{equation*}
    \Gamma(x+1) \sim \sqrt{2\pi x}\cdot x^x \cdot e^{-x} \quad \text{for} \quad x \rightarrow \infty.
\end{equation*}
Writing $x-1$ instead of $x$, it reads:

\begin{equation*}
    \Gamma(x) \sim  \sqrt{2\pi}(x-1)^{\frac{1}{2}}\cdot (x-1)^{x-1}e^{-(x-1)} \quad \text{for} \quad x \rightarrow \infty.
\end{equation*}
Let us apply the Stirling formula in this form to our case, i.e., the formula of Theorem \ref{Theorem: Beziehung}. 
%
A direct application gives:

\begin{equation}\label{eq: Final Asymptotic}
    \Gamma_E\left(x\right)\sim \dfrac{b^x}{\Gamma\left(\frac{a}{b}\right)}\cdot \sqrt{2\pi}\left(x-1-\dfrac{a}{b}\right)^{\frac{1}{2}}\left(x+\dfrac{a}{b}-1\right)^{x+\frac{a}{b}-1}e^{-\left(x+\frac{a}{b}-1\right)}.
\end{equation}
If we massage the expression on the right-hand side into a form resembling Euler's expression (\ref{eq: Asymptotiken}), we end up with the asymptotic:

\begin{equation*}
    \Gamma_E(x) \sim \dfrac{\sqrt{2\pi}}{\Gamma\left(\frac{a}{b}\right)}\cdot b^{-\frac{1}{2}-\frac{a}{b}+1}\left(a+bx-b\right)^{-\frac{1}{2}+x+\frac{a}{b}}e^{-x} \cdot e^{-\frac{a}{b}+1},
\end{equation*}
and comparison to the Eulerian expression reveals the constant $A$ to be:

\begin{equation*}
    A= \dfrac{\sqrt{2\pi}}{\Gamma \left(\frac{a}{b}\right)}\cdot e^{-\frac{a}{b}+1}\cdot b^{\frac{1}{2}-\frac{a}{b}}.
\end{equation*}
Thus, the task given by G. Faber is solved.

\subsection{Final Remark}
\label{subsec: Final Remark}

 Since we only used methods and results already discovered by Euler without any general changes -- we only used modern mathematical language --, we achieved our goal to show that Euler himself could have arrived at the solution to Faber's task and he probably would not have been surprised by the result.
 
 A modern solution of Faber's task would quite possibly use the saddle point method, once Theorem \ref{Theorem: Integraldarstellung} is established. This way, one would not have to rewrite the integral in terms of the $\Gamma$-function and apply the Stirling formula, but would arrive at the final asymptotic asymptotic expansion (\ref{eq: Final Asymptotic}) immediately. But the saddle point method was invented after Euler. The origin of the method is due to Laplace \cite{La74}. Furthermore, we wish to add that Euler's assumption that an equation such as (\ref{eq: Funktionalgleichung}) is given as an integral (\ref{eq: Ansatz}) is correct, can be explained by the theory of the Mellin-transform, which was only introduced by Mellin in his paper \cite{Me96} in 1895, over 100 years after Euler's death in 1783. 

\section{Acknowledgement}
\label{sec: Acknowledgement}

This author of this paper is supported by the Euler-Kreis Mainz. The author
especially wants to thank Prof. T. Sauer, Johannes Gutenberg Universität
Mainz, both for very helpful suggestions concerning the presentation of the subject
and for proof-reading and revising the text.

\end{document}